\documentclass[a4paper]{amsart}
\usepackage{amsmath,amsfonts,amsthm}
\usepackage[latin1]{inputenc}
\usepackage{graphicx}

\usepackage{psfrag}
\usepackage{color}
\usepackage[round]{natbib}

\newtheorem{theorem}{Theorem}
\newtheorem{lemma}{Lemma}
\theoremstyle{definition}
\newtheorem{remark}{Remark}

\newcommand{\sigmaopt}{}
\newcommand{\sigmaoptk}{}

\begin{document}

\title[]{Convergence of an Adaptive Approximation Scheme for the Wiener Process}

\author[M. Brod\'en]{Mats Brod\'en}
\address{Centre for Mathematical Sciences\\ Lund University\\ Box 118\\ 221 00 Lund, Sweden}
\email{matsb@maths.lth.se}

\author[M. Wiktorsson]{Magnus Wiktorsson}
\address{Centre for Mathematical Sciences\\ Lund University\\ Box 118\\ 221 00 Lund, Sweden}
\email{magnusw@maths.lth.se}

\keywords{Discretization error, convergence in distribution, triangular distribution}

\subjclass[2000]{60F05, 60G15}

\begin{abstract}
The problem of approximating/tracking the value of a Wiener process is considered. The discretization points are placed at times when the value of the process differs from the approximation by some amount, here denoted by $\eta$. It is found that the limiting difference, as $\eta$ goes to $0$, between the approximation and the value of the process normalized with $\eta$ converges in distribution to a triangularly distributed random variable.
\end{abstract}

\maketitle

\section{Introduction and preliminaries}
An adaptive approximation scheme of the Wiener process is considered. The discretization points are placed at times when the value of the true process differs from the approximation by some amount, here denoted by $\eta$. This can be seen as a control problem where we want to track the true value of the process with our approximation, and where both the process and its approximation are fully observable. The approximation strategy presented here may be feasible when discretization is associated with some cost that should be kept low. Examples of related problems is that of discrete time hedging of derivative contracts in financial markets (see e.g. \citet{Geiss_Geiss:2006}) and certain space-time discretization schemes of stochastic differential equations (see e.g. \citet{Milstein_Tretyakov:1999}).

Let $X^{\sigmaopt}$ be a diffusion process defined by $X^{\sigmaopt}_t=\sigma W_t$, where $W$ denotes a one dimensional standard Wiener process. Define, for some $\eta>0$, a sequence of stopping times $\{t^{\sigmaoptk \eta}_{i} \}_{i \geq 0}$ by
\begin{align*}
t^{\sigmaoptk \eta}_{i+1} = \inf \{t>t^{\sigmaoptk \eta}_{i} \, | \, |X^{\sigmaopt}_t-X^{\sigmaopt}_{t^{\sigmaoptk \eta}_{i}}|=\eta \} \, ,
\end{align*}
where $t^{\sigmaoptk \eta}_{0}=0$. The components of the sequence $t^{\sigmaoptk \eta}$ may be seen as epochs of the renewal process $N^{\sigmaoptk \eta}$ defined by $N^{\sigmaoptk \eta}_t=\sup\{i: t^{\sigmaoptk \eta}_{i} \leq t \}$. Furthermore, let the sequence $\{\tau^{\sigmaoptk \eta}_{i} \}_{i \geq 1}$ of interarrival times be defined by $\tau_i^{\sigmaoptk \eta}=t_{i}^{\sigmaoptk \eta}-t_{i-1}^{\sigmaoptk \eta}$, and define the renewal-reward process $\varphi$ by $\varphi^{\sigmaoptk \eta}_t:=\sum_{i=1}^{N^{\sigmaoptk \eta}_t} \tau_i^{\sigmaoptk \eta}$. The process $X^{\sigmaopt}_{\varphi^{\sigmaoptk \eta}_t}$ may also be seen as a renewal-reward process, but with a reward that takes the values $-\eta$ and $\eta$ with equal probability.

The aim of this work is to investigate the asymptotic behavior of $(X^{\sigmaopt}_t-X^{\sigmaopt}_{\varphi^{\sigmaoptk \eta}_t})/\eta$ as $\eta$ approaches $0$. It will be seen that this quantity converges, pointwise for each $t>0$, in distribution to a stochastic variable which is triagularly distributed.

Before we end this section we will state some resluts regarding barrier crossings and renewal processes. The main result is presented in Section~\ref{sec:main_res}. In Section~\ref{sec:num_res} we perform a simulation study and investigate the transition to the limiting distribution.

\subsection{The Wiener process with two absorbing barriers}
Since the components of the sequence $\{\tau^{\sigmaoptk \eta}_{i} \}_{i \geq 1}$ are independent and identically distributed, we will let $\tau^{\sigmaoptk \eta}$ denote a stochastic variable with the same properties as these $\tau^{\sigmaoptk \eta}_{i}$'s, and which may be characterized by
\begin{align*}
\tau^{\sigmaoptk \eta} = \inf \{ t>0 \, | \, |X^{\sigmaopt}_t| = \eta \} \, . 
\end{align*}

Now, consider the process $X^{\sigmaopt}$ absorbed in $-\eta$ and $\eta$, that is $X^{\sigmaopt}_{t \wedge \tau}$. The transition density of this process, from $X_0=0$, may be represented by (see \citet{Cox_Miller:1965})
\begin{align}
p^{\sigmaoptk \eta}(t,x) = \sum_{k=1}^{\infty} \frac{1}{\eta^2} e^{-\frac{1}{2} \left( \frac{k \sigma \pi}{2 \eta} \right)^2 t } \sin \left( \frac{k \pi}{2} \right) \sin \left( \frac{k \pi (x+\eta)}{2 \eta} \right) , \label{eqn:pfct}
\end{align}
for all $(t,x) \in (0,\infty) \times [-\eta,\eta]$. This transition density may also be expressed as an infinite sum over Gaussian kernels (see \citet{Cox_Miller:1965})
\begin{align}
p^{\sigmaoptk \eta}(t,x) = \sum_{k=-\infty}^{\infty} \frac{1}{\sqrt{2 \pi \sigma^2 t}} \left(e^{-\frac{(x-4 k \eta)^2}{2 \sigma^2 t}}-e^{-\frac{(x-2\eta+4 k \eta)^2}{2 \sigma^2 t}} \right) . \label{eqn:pfctG}
\end{align}
for all $(t,x) \in [0,\infty) \times [-\eta,\eta]$.

\begin{lemma} \label{lem:ptrip}
The integral of $p^{1}(t,x)$ 
\begin{enumerate}
\item[a)] with respect to $t$ over the interval $[a,b] \subset [0,\infty)$ may be represented as
\begin{align*}
\int_a^b p^{\sigmaoptk \eta}(t,x) \text{d} t = \sum_{k=1}^{\infty} \int_a^b  \frac{1}{\eta^2} e^{-\frac{1}{2} \left( \frac{k \sigma \pi}{2 \eta} \right)^2 t } \sin \left( \frac{k \pi}{2} \right) \sin \left( \frac{k \pi (x+\eta)}{2 \eta} \right) \text{d} t ,
\end{align*}
for all $x \in [-\eta,\eta]$.
\item[b)] with respect to $x$ over the interval $[a,b] \subset [-\eta,\eta]$ may be represented as
\begin{align}
\int_a^b p^{\sigmaoptk \eta}(t,x) \text{d} x = \sum_{k=-\infty}^{\infty} \int_a^b \frac{1}{\eta^2} e^{-\frac{1}{2} \left( \frac{k \sigma \pi}{2 \eta} \right)^2 t } \sin \left( \frac{k \pi}{2} \right) \sin \left( \frac{k \pi (x+\eta)}{2 \eta} \right) \text{d} x  , \label{eqn:repF}
\end{align}
for all $t \in (0,\infty)$, or as
\begin{align}
\int_a^b p^{\sigmaoptk \eta}(t,x) \text{d} x = \sum_{k=-\infty}^{\infty} \int_a^b \frac{1}{\sqrt{2 \pi \sigma^2 t}} \left(e^{-\frac{(x-4 k \eta)^2}{2 \sigma^2 t}}-e^{-\frac{(x-2+4 k \eta)^2}{2 \sigma^2 t}} \right) \text{d} x  , \label{eqn:repG}
\end{align}
for all $t \in [0,\infty)$.
\end{enumerate}
\end{lemma}
\begin{proof} \textit{a)} Define the functions $g^F_k$ and $G^F_n$by 
\begin{align*}
g^F_k(t,x) =  \frac{e^{-\frac{1}{2} \left( \frac{k \sigma \pi}{2 \eta} \right)^2 t }}{\eta^2} \sin \left( \frac{k \pi}{2} \right) \sin \left( \frac{k \pi (x+\eta)}{2 \eta} \right),
\end{align*}
and $G^F_n(t,x)=\sum_{k=1}^n g^F_k(t,x)$ then $\lim_{n \uparrow \infty}G^F_n(t,x)=p^{\eta}(t,x)$. Since $g_k^F(t,0) \geq 0$ it follows that
$$
0 \leq G_{n}^F(t,0) \leq G_{n+1}^F(t,0),
$$
and consequently by Lebesgues monotone convergence thorem
$$
\int_a^b \lim_{n \uparrow \infty} G_n(t,0) \text{d} t = \lim_{n \uparrow \infty} \int_a^b  G_n(t,0) \text{d} t.
$$
Extending the integral we get
$$
\lim_{n \uparrow \infty} \int_a^b  G_n(t,0) \text{d} t \leq \lim_{n \uparrow \infty} \int_0^\infty  G_n(t,0) \text{d} t.
$$
Moving the integral inside of the sum in $G_n(t,0)$ and performing the integration over $\mathbb{R}_+$ we get the sum $\lim_{n \uparrow \infty} \sum_{k=1}^n 8/(k^2 \pi^2 \sigma^2)=4/(3 \sigma^2)$, and hence $\lim_{n \uparrow \infty} \int_0^\infty  G_n(t,0) \text{d} t < \infty$. Since $|\sin(k \pi/2) \sin(k \pi (x+\eta)/(2 \eta))| \leq 1$  it holds that $|g_k^F(t,x)| \leq g_k^F(t,0)$ which implies that $|G_n^F(t,x)| \leq G_n^F(t,0)$. Since $G_n^F(t,0)$ is bounded by $\lim_{n \uparrow \infty} G_n^F(t,0)$ the function $G_n^F(t,x)$ is dominated by the integrable function $\lim_{n \uparrow \infty}$ $G_n^F(t,0)$ and by the dominated convergence theorem it follows that
$$
\int_a^b \lim_{n \uparrow \infty} G_n(t,x) \text{d} t = \lim_{n \uparrow \infty} \int_a^b  G_n(t,x) \text{d} t.
$$
Moving the integral inside of the sum on the right hand side the claim is proved.

\textit{b)} \textit{Eqn. \eqref{eqn:repF}} From the proof of \textit{a)} we know that $G_n^F(t,x) \leq p^\eta(t,0) = \lim_{n \uparrow \infty}$ $G_n(t,0)$ which is bounded for every $t>0$. Since the set $[a,b]$ is bounded (i.e. $[a,b] \subset [-\eta,\eta]$), the claim now follows from the bounded convergence thorem.

\textit{Eqn. \eqref{eqn:repG}} Define the functions $g_k^G$ and $G_n^G$ by
\begin{align*}
g_k^G(t,x) = \frac{e^{-\frac{(x-k)^2}{2 \sigma^2 t}}}{\sqrt{2 \pi \sigma^2 t}} && \text{and} && G^G_n(t,x)=\sum_{k=-n}^n (g^G_{4k\eta}(t,x)-g^G_{2-4k\eta}(t,x)),
\end{align*}
then $\lim_{n \uparrow \infty} G^G_n(t,x)=$ $p^\eta(t,x)$. The function $G_n^G$ may be decomposed as 
$$
G_n^G(t,x)=G_n^{G,1}(t,x)+G_n^{G,2}(t,x),
$$
where
$$
G_n^{G,1}(t,x)=\sum_{k=0}^n (g^G_{4k\eta}(t,x)-g^G_{4k\eta+2}(t,x))
$$
and
$$
G_n^{G,2}(t,x)=\sum_{k=-1}^{-n}(g^G_{4k\eta}(t,x)-g^G_{4k\eta+2}(t,x)). 
$$
Since each term in $G_n^{G,1}$ is positive and each term in $G_n^{G,2}$ is negative it holds that
\begin{align*}
0 \leq G_n^{G,1}(t,x) \leq G_{n+1}^{G,1}(t,x) && \text{and} && 0 \geq G_{n}^{G,2}(t,x) \geq G_{n+1}^{G,2}(t,x).
\end{align*}
The claim now follows by Lebesgues monotone convergence theorem.
\end{proof}

\begin{lemma} \label{lem:ptri}
It holds that
\begin{equation*}
\sigma^2 \int_0^{\infty} p^{\sigmaoptk  1}(t,x) dt=(1-|x|)^+ \, .
\end{equation*}
\end{lemma}
\begin{proof} From Lemma \ref{lem:ptrip} \textit{a)} we have that
\begin{equation*}
\int_0^{\infty} p^{\sigmaoptk 1}(t,x) \text{d}t =  \frac{8}{\pi^2 \sigma^2} \sum_{k=1}^{\infty} \frac{1}{k} \sin \left( \frac{k \pi}{2} \right) \frac{1}{k} \sin \left( \frac{ k \pi (x+1)}{2} \right) \, .
\end{equation*}
The idea is to find a function that can be expressed as a series which corresponds to the above sum. Let $s_1=1/2$ and $s_2=(x+1)/2$, then
\begin{equation*}
\frac{\pi^2 \sigma^2}{8} \int_0^{\infty} p^{\sigmaoptk 1}(t,x) \text{d}t = \sum_{k=1}^{\infty} \frac{1}{k} \sin \left( k \pi s_1 \right) \frac{1}{k} \sin \left( k \pi s_2 \right) \, .
\end{equation*}
Define the function $h_s$ by
\begin{equation*}
h_s(x) = \left\{
\begin{array}{ll}
0 \, , & 0 \leq |x| \leq s \, , \\
1 \, , & s < |x| \leq 1  \, .
\end{array} \right.
\end{equation*}
The Fourier Cosine coefficients of $h_s$ are given by
\begin{equation*}
\begin{split}
c_0 & = \int_0^1 h_s(x) \text{d}x = 1-s \, , \\
a_k & = 2 \int_0^1 \cos(k \pi x) h_s(x) \text{d}x = -\frac{2 \sin(\pi k s)}{\pi k} \, .
\end{split} 
\end{equation*}
Applying Parseval's formula yields
\begin{equation*}
\int_0^1 h_{s_1}(x) h_{s_2}(x) \text{d}x = 2 \sum_{k=1}^{\infty} \frac{\sin(\pi k s_1)}{\pi k} \frac{\sin(\pi k s_2)}{\pi k} + (1-s_1)(1-s_2) \, .
\end{equation*}
Assume that $x \in [0,1]$, then $0 \leq s_1 \leq s_2 \leq 1$ and
\begin{align*}
\sum_{k=1}^{\infty} \frac{2}{\pi^2 k^2} \sin(\pi k s_1) \sin(\pi k s_2) &  = 1-s_2-(1-s_1)(1-s_2) \\
& = s_1(1-s_2) \, .
\end{align*}
Thus
\begin{align*}
\sigma^2 \int_0^{\infty} p^{\sigmaoptk 1}(t,x) \text{d}t & = 4  \sum_{k=1}^{\infty} \frac{2}{k^2 \pi^2} \sin \left( \frac{k \pi}{2} \right) \sin \left( \frac{k \pi (x+1)}{2} \right) \\
& =  (1-x) \, .
\end{align*}
Repeating the argument with $x \in [-1,0]$ yields the result.
\end{proof}

One important property in the theory of renewal processes is that of direct Riemann integrability of a function. A function function $H(\cdot)$ is said to be \textit{directly Riemann integrable} over $[0,\infty)$ if for any $h>0$, the normalized sums
\begin{align*}
h \sum_{n=1}^\infty \inf_{0 \leq \delta \leq h} H(nh-\delta) && \text{and} && h \sum_{n=1}^\infty \sup_{0 \leq \delta \leq h} H(nh-\delta),
\end{align*}
converge to a common finite limit as $h \downarrow 0$ (see chapter 4.4 in \citet{Daley_VereJones:1988}).

\begin{lemma} \label{lem:pdri}
The function $p^1(t,x)$ is directly Riemann integrable with respect to $t$ for each $x \in [-1,1]$.
\end{lemma}
\begin{proof}
We will start by considering the case when $x=0$. The function $p^1(t,0)$ is directly Riemann integrable if $p^1(t,0)$ is nonegative, monotonically decreasing and Lebesgue integrable (see chapter 4.4 in \citet{Daley_VereJones:1988}). Since each term in the representation \eqref{eqn:repF} is nonegative and monotonically decreasing for $x=0$ so is $p^1(t,0)$, and by Lemma \ref{lem:ptri} the integral of $p^1(t,0)$ over $[0,\infty]$ is given by $\int_0^\infty p^1(t,0) \text{d}t=1$ and thus $p^1(t,0)$ is Lebesgue integrable which proves that $p^1(t,0)$ is directly Riemann integrable.

Next let $x \in [-1,1] \setminus \{0\}$. The function $p(t,x)$ is directly Riemann integrable with respect to $t$ if $p^1(t,x) \geq 0$, $p(t,x)$ is uniformly continuous in $t$ and bounded from above by a monotonically decreasing integrable function (see chapter 4.4 in \citet{Daley_VereJones:1988}). Since $p(t,x)$ is a probability distribution for each $t$ it is clear that $p(t,x) \geq 0$. To show uniform continuity we will split the interval $[0,\infty)$ into two parts, say $[0,1]$ and $[1,\infty)$, and show that $p^1(t,x)$ is uniformly continuous on each part. For the interval $[0,1]$ we will use the representation \eqref{eqn:repG}. Let $g_k^G$ and $G_n^G$ be defined as in the proof of Lemma \ref{lem:ptrip}. It is clear that each $g_k^G$ is uniformly continuous in $t$ and thus also $G_n^G$ is uniformly continuous for each $n<\infty$. If we can shown that $G_n^G(t,x)$ for each  $x \in [-1,1] \setminus \{0\}$ converges uniformly with respect to $t$ over $[0,1]$ as $n \uparrow \infty$, then also the limit $p(t,x)$ will be uniformly continuous. Rewrite $G_n^G$ as $G_n^G(t,x)=\sum_{k=0}^{n} \tilde{g}_k^G(t,x)$ where $\tilde{g}_0^G(t,x)=g^G_{0}(t,x)-g^G_{2}(t,x)$  and
\begin{align*}
\tilde{g}_k^G(t,x) = g^G_{4k}(t,x)-g^G_{2-4k}(t,x)+g^G_{-4k}(t,x)-g^G_{2+4k}(t,x), \text{ for } k \geq 1.
\end{align*}
According to Weierstrass M-test, if there is a series of constants $M_k$ such that $\sum_{k=0}^{\infty} M_k$ is convergent and $|\tilde{g}_k^G(t,x)| \leq M_k$ for all $t \in [0,1]$ then $G_n^G$ converges uniformly in $[0,1]$ as $n \uparrow \infty$. The functions $g_k(t,x)$ attains its maximum at $t=(x-k)^2/\sigma^2 \wedge 1$ for $t \in [0,1]$, and thus $g_k(t,x) \leq g_k((x-k)^2/\sigma^2 \wedge 1,x)$. The function $g_0(x^2/\sigma^2 \wedge 1,x)$ is bounded and it is easily seen that the functions $g_k^G$ may be bounded by $C/(1+k^2)$, for some bounded constant $C$, and which is clearly convergent. Hence, for each $x \in [-1,1] \setminus \{0\}$, $p(\cdot,x)$ is uniformly continuous in $[0,1]$. To show uniform continuity in $[1,\infty)$ we will use the representation \eqref{eqn:repF}. Let $t \geq 1$, then
\begin{align*}
|p^1(t+\delta,x)-p^1(t,x)| & \leq \sum_{k=1}^{\infty} e^{-\frac{k^2 \sigma^2 \pi^2}{8}t} |e^{-\frac{k^2 \sigma^2 \pi^2}{8}\delta}-1| \\
&\leq \sum_{k=1}^{\infty} \frac{8^2}{k^4 \sigma^4 \pi^4} \frac{k^2 \sigma^2 \pi^2}{8} \delta =  \delta \frac{3}{4 \sigma^2}
\end{align*}
where we used the inequalites $e^{-y} \leq y^{-2}$ and $|e^{-y}-1|\leq y$ which holds for $y \geq 0$. Hence for every $\epsilon>0$ we may chose $\delta$ such that $\delta < 4 \sigma^2 \epsilon/3$ which holds for every $t$ in $[1,\infty)$. Hence $p^1(\cdot,x)$ is also uniformly continuous in $[1,\infty)$, which together with the previous result yields that $p(\cdot,x)$ is uniformly continuous in $[0,\infty)$. In the proof of Lemma \ref{lem:ptrip} we showed that $p(t,x) \leq p(t,0)$, and that $p(t,0)$ is a monotonically decreasing Lebesgue integrable function. Hence, $p(t,x)$ is also directly Riemann integrable with respect to $t$ for $x \in [-1,1] \setminus \{0\}$, which together with the first result of this proof yeilds that $p(t,x)$ is directly Riemann integrable for $x \in [-1,1]$.
\end{proof}

The next two lemmas regards properties of the random variable $\tau^{\eta}$ defined earlier in this section. Let $F_{\tau^{\eta}}$ denote the distribution function of $\tau^{\eta}$. Lemma \ref{lem:taudist} states that that $\tau^{\eta}$ has a density, which we will denote by $f_{\tau^{\eta}}$.

\begin{lemma} \label{lem:Etau}
The expectation of $\tau^{\sigmaoptk \eta}$ is given by $E[\tau^{\sigmaoptk \eta}] = \eta^2/\sigma^2$ .
\end{lemma}
\begin{proof} Let $g(x_0)=E[\tau^{\sigmaoptk \eta}]$, where $x_0$ denotes the initial point of the process. The function $g$ satisfies the following ordinary differential equation (see \citet{Cox_Miller:1965})
\begin{align*}
\frac{\sigma^2}{2} \frac{d^2 g}{d x_0^2}(x_0) = -1 , \quad m_1(-\eta)=m_1(\eta)=0 \, .
\end{align*}
The solution to this problem, with $x_0=0$, is given by $g(0)=\eta^2/\sigma^2$, as was to be shown.
\end{proof}

\begin{lemma} \label{lem:taudist}
The random variable $\tau^\eta$ has a density, denoted by $f_{\tau^\eta}$, that may be represented as
\begin{multline*}
f_{\tau^\eta}(t) = \sum_{k=-\infty}^{\infty} \frac{1}{2 t \sqrt{2 \pi \sigma^2 t}} \Biggl( (\eta+4k\eta) e^{-\frac{(\eta+4k\eta)^2}{2 \sigma^2 t}}-(\eta+2-4k\eta)e^{-\frac{(\eta+2-4k\eta)^2}{2 \sigma^2 t}} \\
+(\eta-4k\eta) e^{-\frac{(\eta-4k\eta)^2}{2 \sigma^2 t}}-(\eta-2+4k\eta) e^{-\frac{(\eta-2+4k\eta)^2}{2 \sigma^2 t}}  \Biggr),
\end{multline*}
for all $t \in [0,\infty)$.
\end{lemma}
\begin{proof}
In this proof we will use the representation \eqref{eqn:repG}. Let $g_k^G$ and $G_n^G$ be defined as in the proof of Lemma \ref{lem:ptrip}. By the use of Lemma \ref{lem:ptrip} for $t \in [0,\infty)$
\begin{align*}
P(\tau^{\eta} \leq t) = 1-\sum_{k=-\infty}^{\infty} \int_{-\eta}^{\eta} (g_{4k\eta}^G(t,x)-g_{2-4k\eta}^G(t,x)) \text{d} x .
\end{align*}
If each term in the sum above is differentiable on $[0,\infty)$ and 
\begin{align}
\sum_{k=-\infty}^{\infty} \frac{d}{d t} \int_{-\eta}^{\eta} (g_{4k\eta}^G(t,x)-g_{2-4k\eta}^G(t,x)) \text{d} x \label{eqn:dPdtsum}
\end{align}
converges uniformly on $[0,\infty)$ then
\begin{align*}
\frac{d}{dt} P(\tau^{\eta} \leq t) = -\sum_{k=-\infty}^{\infty}  \frac{d}{d t} \int_{-\eta}^{\eta} (g_{4k\eta}^G(t,x)-g_{2-4k\eta}^G(t,x)) \text{d} x .
\end{align*}
Calculating the integral and differentiating with respect to $t$ we get for each term in \eqref{eqn:dPdtsum}
\begin{equation}
\begin{split}
&  \frac{d}{d t} \int_{-\eta}^{\eta} (g_{4k\eta}^G(t,x)-g_{2-4k\eta}^G(t,x)) \text{d} x \\
& = \frac{1}{2 t \sqrt{2 \pi \sigma^2 t}} \Biggl( (\eta+4k\eta) e^{-\frac{(\eta+4k\eta)^2}{2 \sigma^2 t}}-(\eta+2-4k\eta)e^{-\frac{(\eta+2-4k\eta)^2}{2 \sigma^2 t}} \\
& \qquad \qquad \qquad +(\eta-4k\eta) e^{-\frac{(\eta-4k\eta)^2}{2 \sigma^2 t}}-(\eta-2+4k\eta) e^{-\frac{(\eta-2+4k\eta)^2}{2 \sigma^2 t}}  \Biggr), \label{eqn:dPdtterm}
\end{split}
\end{equation}
The maximum of the function $e^{-\frac{(x-k)^2}{2 \sigma^2 t}}/t^{3/2}$ in $[0,\infty)$ is attained at $t=(x-k)^2/(3 \sigma^2)$. For the first term in the expression above we get that
\begin{align*}
\frac{(\eta+4k\eta)}{2 t^{3/2} \sqrt{2 \pi \sigma^2}}e^{-\frac{(\eta+4k\eta)^2}{2 \sigma^2 t}} \leq \left(\frac{3}{2}\right)^{3/2} \frac{\sigma^2 e^{-\frac{3}{2}}}{\eta^2 \sqrt{\pi}} \frac{1}{(1+4k)^{2}},
\end{align*}
which may be bounded by $C/(1+k^2)$, where $C$ is a bounded constant. In a similar manner it can be shown that the rest of the terms in \eqref{eqn:dPdtterm} may also be bounded by $C/(1+k^2)$, and thus
\begin{equation}
\left| \frac{d}{d t} \int_{-\eta}^{\eta} (g_{4k\eta}^G(t,x)-g_{2-4k\eta}^G(t,x)) \text{d} x \right| \leq \frac{4C}{1+k^2}
\end{equation}
Since $\sum_{k=-\infty}^{\infty} 4C/(1+k^2)$ is a convergent series by Wierstrass M-test the sum \eqref{eqn:dPdtsum} converges uniformly on $[0,\infty)$, and hence, the density, $f_{\tau^{\eta}}$, may be represented by the sum \eqref{eqn:dPdtterm}. Since the terms in the sum of \eqref{eqn:dPdtterm} could be bounded by $4C/(1+k^2)$ we have that $|f_{\tau^{\eta}}(t)| \leq 4 C \sum_{k=-\infty}^\infty 1/(1+k^2)<\infty$ which shows that $f_{\tau^{\eta}}(t)$ is bounded in $[0,\infty)$.
\end{proof}

\subsection{Renewal processes}
In this paragraph we will focus on a renewal process denoted by $N$ with idenpendent and identically distributed interarrival times $\{ \tau_i \}_{i \geq 1}$. Define the renewal function $M$ by $M_t=E[N_t]$, and let $\mu$ denote the mean time between renewals, that is $\mu = E[\tau_i]$, which holds for all $i \geq 1$. Next, we will state the key renewal theorem that will be needed later on.

\begin{lemma}[Key renewal theorem] \label{lem:krt}
If $H(\cdot)$ is a directly Riemann-integrable function then
\begin{align*}
\lim_{t \rightarrow \infty} \int_0^t H(t-x) dM(x) = \frac{1}{\mu} \int_0^{\infty} H(x) dx \, .
\end{align*}
\end{lemma}
\begin{proof}
See e.g. \citet{Daley_VereJones:1988}.
\end{proof}

Let $F_{\tau}$ denote the common distribution function of the stochastic variables $\tau_i$. Since the components of $\{ \tau_i \}_{i\geq 1}$ are idependent and identically distributed the distribution function of the sum $\sum_{i=1}^k \tau_i$ may be represented by the $k$-fold convolution of $F_{\tau}$ (here denoted $F^{*k}_{\tau}$), i.e.
\begin{align*}
P\left(\sum_{i=1}^k \tau_i < t \right) = F^{*k}_{\tau}(t) \, .
\end{align*}

\begin{lemma}[Theorem 5.4 in \citet{Heyman_Sobel:1982}]
There exists a one-to-one correspondence between $F_{\tau}$ and $M$, and $M$ has the representation
\begin{align*}
M_t = \sum_{k=1}^{\infty} F^{*k}_{\tau}(t) \, .
\end{align*}
\end{lemma}

Under the assumption that $F_{\tau}$ has a density (here denoted $f_{\tau}$) we have that
\begin{align*}
f^{*k}_{\tau}(t) = \frac{d}{dt} F^{*k}_{\tau}(t) \, ,
\end{align*}
where $f^{*k}_{\tau}$ is the $k$-th convolution of the density function $f_{\tau}$. We may now define the renewal density $m$ by
\begin{align}
m_t := \frac{d}{dt} M_t  = \sum_{k=1}^{\infty} f^{*k}_{\tau}(t) \, . \label{eqn:m_expr}
\end{align}

\section{Main result} \label{sec:main_res}
In this section we state and prove the main result of this paper. To ease the notation in the proof we will let $Z^{\sigmaoptk \eta}_t=X^{\sigmaoptk \eta}_t-X^{\sigmaoptk \eta}_{\varphi^{\sigmaoptk \eta}_t}$.

\begin{theorem} \label{thm:main}
Fix a point $t>0$, then
\begin{align*}
\frac{1}{\eta} \left( X^{\sigmaopt}_t-X^{\sigmaopt}_{\varphi^{\sigmaoptk \eta}_t}  \right) \stackrel{d}{\longrightarrow} \Lambda \quad \text{as } \eta \rightarrow 0 \, , 
\end{align*}
where $\Lambda$ is a stochastic variable with density function given by
\begin{equation*}
f_{\Lambda}(z) = (1-|z|)^+ \, .
\end{equation*}
\end{theorem}

\begin{proof}
Denote by $Y^{\sigmaoptk \eta}_t(u)$ the quantity
\begin{equation*}
Y^{\sigmaoptk \eta}_t(u)=X^{\sigmaopt}_t-X^{\sigmaopt}_{\varphi^{\sigmaoptk \eta}_t}| \{ t-\varphi^{\sigmaoptk \eta}_t=u \} \, .
\end{equation*}
Because of the time homogeneity of the process $X^{\sigmaopt}$ the following equality in distribution holds
\begin{equation*}
X^{\sigmaopt}_t-X^{\sigmaopt}_{\varphi^{\sigmaoptk \eta}_t}|\{ t-\varphi^{\sigmaoptk \eta}_t=u \} \stackrel{d}{=} X^{\sigmaopt}_u|\{ | X^{\sigmaopt}_s | < \eta , \, 0 \leq s \leq u \} \, .
\end{equation*}
Consequently the density function of $Y^{\sigmaoptk \eta}_t(u)$ can be expressed as
\begin{equation*}
f_{Y^{\sigmaoptk \eta}_t(u)}(y)=\frac{p^{\sigmaoptk \eta}(u,y)}{P(\tau^{\sigmaoptk \eta} > u)} \, ,
\end{equation*}
The distribution function of $Z^{\sigmaoptk \eta}_t$ is given by
\begin{equation*}
f_{Z^{\sigmaoptk \eta}_t}(z) = \int_0^t f_{Y^{\sigmaoptk \eta}_t(u)}(z) \text{d}F_{t-\varphi^{\sigmaoptk \eta}_t}(u) \, ,
\end{equation*}
where
\begin{align*}
& \text{d}F_{t-\varphi^{\sigmaoptk \eta}_t}(u) = \left\{ \delta(u-t) P(\tau^{\sigmaoptk \eta}>t)+\sum_{k=1}^{\infty} \frac{\partial}{\partial u} P(t-\varphi^{\sigmaoptk \eta}_t \leq u, \, N^{\sigmaoptk \eta}_{t}=k) \right\} \text{d}u \, .
\end{align*}
The probability in the last term of the above expression can be rewritten as
\begin{equation*}
\begin{split}
P(t-\varphi^{\sigmaoptk \eta}_t \leq u, \, N^{\sigmaoptk \eta}_{t}=k) & = P \left( t - \sum_{j=1}^{k} \tau^{\sigmaoptk \eta}_{j} \leq u , \, \sum_{j=1}^{k} \tau^{\sigmaoptk \eta}_{j} < t < \sum_{j=1}^{k} \tau^{\sigmaoptk \eta}_{j} + \tau^{\sigmaoptk \eta}_{k+1} \right) \\
& = P \left( t - \sum_{j=1}^{k} \tau^{\sigmaoptk \eta}_{j} \leq u , \, 0 < t - \sum_{j=1}^{k} \tau^{\sigmaoptk \eta}_{j} < \tau^{\sigmaoptk \eta}_{k+1} \right) \\
& = \int_{t-u}^{\infty} \int_{t-v}^{\infty} f_{\tau^{\sigmaoptk \eta}}^{* k} (v) f_{\tau^{\sigmaoptk \eta}}(z) \text{d}z \, \text{d}v \, ,
\end{split}
\end{equation*}
where $f_{\tau^{\sigmaoptk \eta}}$ (which exists due to Lemma \ref{lem:taudist})  is the density function of $\tau^{\sigmaoptk \eta}$, and $f^{*k}_{\tau^{\sigmaoptk \eta}}$ denotes the $k$-th convolution of  $f_{\tau^{\sigmaoptk \eta}}$. Differentiating the above expression with respect to $u$ yields
\begin{equation*}
\begin{split}
\frac{\partial}{\partial u} \left( \int_{t-u}^{\infty} \int_{t-v}^{\infty} f_{\tau^{\sigmaoptk \eta}}^{* k} (v) f_{\tau^{\sigmaoptk \eta}}(z) dz \, dv \right) & = \int_{u}^{\infty} f_{\tau^{\sigmaoptk \eta}}^{* k} (t-u) f_{\tau^{\sigmaoptk \eta}}(z) \text{d}z \\
& = f_{\tau^{\sigmaoptk \eta}}^{* k} (t-u) P(\tau^{\sigmaoptk \eta}>u) \, .
\end{split}
\end{equation*}
This gives us that
\begin{equation*}
\text{d}F_{t-\varphi^{\sigmaoptk \eta}_t}(u) = \left\{ \delta(u-t) P(\tau^{\sigmaoptk \eta}>t)+\sum_{k=1}^{\infty} f_{\tau^{\sigmaoptk \eta}}^{* k} (t-u) P(\tau^{\sigmaoptk \eta}>u) \right\} \text{d}u \, .
\end{equation*}
Using the scaling property of the Brownian motion the following two relations are easily deduced
\begin{align*}
P(\tau^{\sigmaoptk \eta}>t)=P(\tau^{\sigmaoptk 1}>t/\eta^2) \quad \text{and} \quad \frac{1}{\eta} Y^{\sigmaoptk \eta}_t(u) \stackrel{d}{=} Y^{\sigmaoptk 1}_t(u/\eta^2) \, .
\end{align*}
The first of the two relations above yields
\begin{align*}
f_{\tau^{\sigmaoptk \eta}}(t)=-\frac{\text{d}}{\text{d}t}P(\tau^{\sigmaoptk \eta}>t)=-\frac{\text{d}}{\text{d}t}P(\tau^{\sigmaoptk 1}>t/\eta^2) = \frac{1}{\eta^2}f_{\tau^{\sigmaoptk 1}}(t/\eta^2) \, ,
\end{align*}
and consequently
\begin{align*}
& \text{d}F_{t-\varphi^{\sigmaoptk \eta}_t}(u)  = \left\{ \delta(u-t) P(\tau^{\sigmaoptk 1}>t/\eta^2)+\sum_{k=1}^{\infty} \frac{1}{\eta^2} f_{\tau^{\sigmaoptk 1}}^{* k} \left( \frac{t-u}{\eta^2} \right) P(\tau^{\sigmaoptk 1}>u/\eta^2) \right\} \text{d}u  \, .
\end{align*}
The relation $Y^{\sigmaoptk \eta}_t(u)/\eta \stackrel{d}{=} Y^{\sigmaoptk 1}_t(u/\eta^2)$  yields
\begin{align*}
& \int_0^t f_{Y^{\sigmaoptk \eta}_t(u)/\eta}(y) \text{d}F_{t-\varphi^{\sigmaoptk \eta}_t}(u) = \int_0^t f_{Y^{\sigmaoptk 1}_t(u/\eta^2)}(y) \text{d}F_{t-\varphi^{\sigmaoptk \eta}_t}(u) \, ,
\end{align*}
and thus
\begin{align*}
f_{Z^{\sigmaoptk \eta}_t/\eta}(z)  & = \int_0^t \frac{p^{\sigmaoptk  1}(u/\eta^2,z)}{P(\tau^{\sigmaoptk 1}>u/\eta^2)} \delta(u-t) P(\tau^{\sigmaoptk 1}>t/\eta^2) \text{d}u \\
& \quad + \int_0^t \frac{p^{\sigmaoptk  1}(u/\eta^2,z)}{P(\tau^{\sigmaoptk 1}>u/\eta^2)} \sum_{k=1}^{\infty} \frac{1}{\eta^2} f_{\tau^{\sigmaoptk 1}}^{* k} \left( \frac{t-u}{\eta^2} \right) P(\tau^{\sigmaoptk 1}>u/\eta^2) \text{d}u \\
& = p^{\sigmaoptk  1}(t/\eta^2,z) + \int_0^t p^{\sigmaoptk  1}(u/\eta^2,z) \sum_{k=1}^{\infty} \frac{1}{\eta^2} f_{\tau^{\sigmaoptk 1}}^{* k} \left( \frac{t-u}{\eta^2} \right)  \text{d}u \, .
\end{align*}
Now, by a change of variables ($v=(t-u)/\eta^2$)
\begin{align*}
f_{Z^{\sigmaoptk \eta}_t/\eta}(z) =  p^{\sigmaoptk  1}(t/\eta^2,z) + \int_0^{t/\eta^2} p^{\sigmaoptk  1} \left( \frac{t}{\eta^2}-v , z \right) \sum_{k=1}^{\infty} f_{\tau^{\sigmaoptk 1}}^{* k} \left( v \right) \text{d}v \, .
\end{align*}
Since
\begin{align*}
|p^1(t/\eta^2,x)| \leq \sum_{k=1}^{\infty} \frac{8 \eta^2}{k^2 \sigma^2 \pi^2}=\frac{4 \eta^2}{3 \sigma^2},
\end{align*}
we have that $\lim_{\eta \rightarrow 0} p^{\sigmaoptk \eta}(y,t/\eta^2) = 0$. For the second term we have using \eqref{eqn:m_expr}, Lemma \ref{lem:Etau} and Lemma \ref{lem:krt} (which applicable since $p(t,x)$ is a directly Riemann integrable function due to Lemma \ref{lem:pdri})
\begin{align*}
\lim_{\eta \rightarrow 0} \int_0^{t/\eta^2} p^{\sigmaoptk  1} \left( y,\frac{t}{\eta^2}-v \right) \sum_{k=1}^{\infty} f_{\tau^{\sigmaoptk 1}}^{* k} \left( v \right)  dv  = \sigma^2 \int_0^{\infty} p^{\sigmaoptk 1}(y,u) \text{d}u\, .
\end{align*}
Now by Lemma \ref{lem:ptri}
\begin{equation*}
\lim_{\eta \rightarrow 0} f_{Z^{\sigmaoptk \eta}_t/\eta}(z) = (1-|z|)^+ \, ,
\end{equation*}
as was to be shown.
\end{proof}

\begin{remark}
Note that the limiting distribution does not depend on $\sigma$. This is unlike the case when discretization takes place on an equidistant grid, where $\sigma$ affects the variance of the limiting distribution. Instead, in the case of adaptive approximation, $\sigma$ is related to the expected number of discretization points.
\end{remark}

\begin{remark}
In the proof above all interarrival times $\tau_i^{\sigmaoptk  \eta}$ up to the time $t$ is used in order to characterize the distribution of $t-\varphi_t^{\sigmaoptk  \eta}$. However, we belive that $t-\varphi_t^{\sigmaoptk  \eta}$  may be characterized by the dynamics of the process $X$ in a small region around $X_t$, which would imply that the result of Theorem \ref{thm:main} is a local result. From this and the fact that the diffusion coefficient $\sigma$ is scaled away in the limiting expression of the distribution we conjecture that Theorem \ref{thm:main} would hold for a larger class of stochastic processes such as SDE's. We plan to address this in future research.
\end{remark}

\section{Numerical results} \label{sec:num_res}
In this section the transition of $f_{Z^{\sigmaoptk \eta}_t/\eta}$ as $\eta$ goes from some large value towards zero is investigated. We will argue that for large values of $\eta$ the stochastic variable $Z^{\sigmaoptk \eta}_t/\eta$ is approximately normally distributed, and thus as $\eta$ approaches zero we will see that $f_{Z^{\sigmaoptk \eta}_t/\eta}$ goes from the density of a normally distributed random variable to the density of a triangularly distributed random variable.

A total of $50000$ trajectories of the process $X$ was simulated, over a period from $t=0$ to $t=0.5$, with $\sigma=1$, on a time grid with $200001$ equally spaced points. Trajectories of the approximation $X_{\varphi^{\sigmaoptk \eta}_t}$ were calculated for a number of different values of $\eta$ in the range $[0.5,4.0]$.

Recall, from the proof of Theorem \ref{thm:main}, the expression of the density 
\begin{align}
f_{Z^{\sigmaoptk \eta}_t/\eta}(z) =  p^{\sigmaoptk  1}(t/\eta^2,z) + \int_0^{t/\eta^2} p^{\sigmaoptk  1} \left( \frac{t}{\eta^2}-v,z \right) \sum_{k=1}^{\infty} f_{\tau^{\sigmaoptk 1}}^{* k} \left( v \right)  dv \, . \label{eqn:pdf_Zeta}
\end{align}
It is clear that for large values of $\eta$ it is the first term in \eqref{eqn:pdf_Zeta} that is the dominant one. Thus, in this case the density is approximately the same as the absorbed Wiener process. Furthermore, since $\eta$ was assumed to be large the density of the absorbed Wiener process is approximately the same as the Wiener process without absorbing barriers. Hence, for large $\eta$ we have that 
\begin{align}
f_{Z^{\sigmaoptk \eta}_t/\eta}(z) \approx \frac{\eta}{ \sigma \sqrt{t}} \phi\left( z \frac{\eta}{ \sigma \sqrt{t} } \right) \, , \label{eqn:fnorm}
\end{align}
where $\phi$ denotes the standard normal density function.

In Figure \ref{fig:pdf} the density of $f_{Z^{\sigmaoptk \eta}_t/\eta}$, at $t=0.5$, as we let $\eta$ go from $4.0$ to $0.5$ is depicted. It is seen that when $\eta=4.0$ the distribution is quite close to the normal distribution. For $\eta=0.5$ the distribution on the other hand is quite close to the triangular distribution.
\begin{figure}
\centerline{
{\psfrag{x}[bc][bc]{$z$}
\psfrag{pdf}[bc][bc]{$\text{pdf}$}
\includegraphics[width=0.5\textwidth]{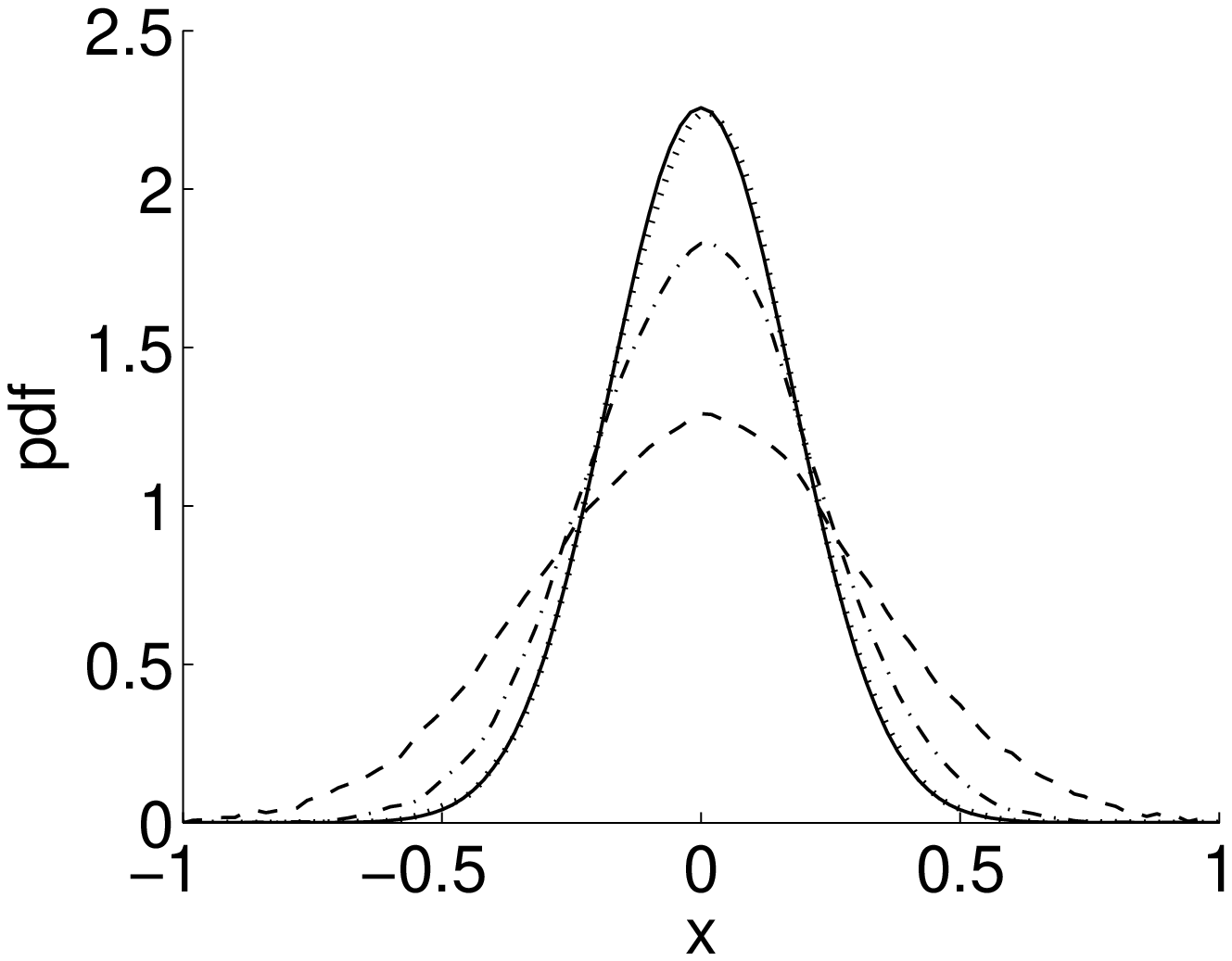}\includegraphics[width=0.5\textwidth]{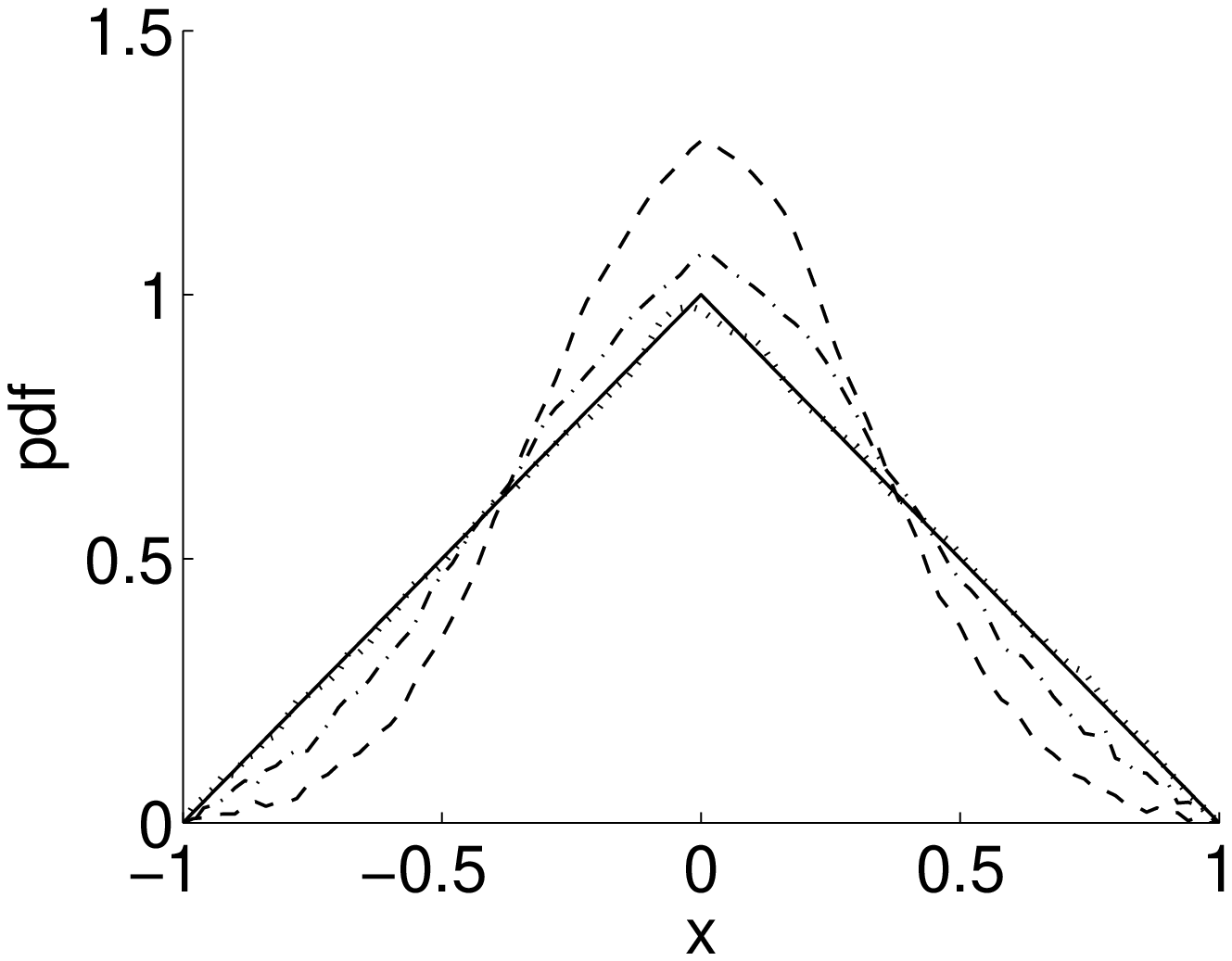}}}
\caption{Left (large values of $\eta$): kernel estimates of $f_{Z^{\sigmaoptk \eta}_t/\eta}(z)$ where $\eta=4.0$ (dotted line), $\eta=3.25$ (dash-dotted line) and $\eta=2.5$ (dashed line), and the Gaussian distribution (solid line). Right (small values of $\eta$): kernel estimates of $f_{Z^{\sigmaoptk \eta}_t/\eta}(z)$ where $\eta=2.5$ (dashed line), $\eta=2.0$ (dash-dotted line) and $\eta=0.5$ (dotted line), and the triangular distribution (solid line). }
\label{fig:pdf}
\end{figure}

To further illustrate the transition from the normal distribution to the triangular distribution we measured the distance in therms of the Wasserstein metric between the, from the Monte Carlo simulation, estimated distribution and these two distributions. The distance between two distributions, with distribution functions $F$ and $G$, in terms of the Wasserstein metric is defined by
\begin{align*}
d_W(F,G) = \int_{\mathbb{R}} |F(x)-G(x)| dx \, .
\end{align*}

In Figure \ref{fig:dist} the Wasserstein distance between the empirical distribution and the triangular distribution as well as the distance between the empirical distribution and the normal distribution \eqref{eqn:fnorm}, at $t=0.5$, as a function of $\eta$ is depicted. Note that in the case of the normal distribution \eqref{eqn:fnorm} not only the empirical distribution but also the normal distribution that we compare with is dependent of $\eta$. It is seen that for $\eta$ smaller than $1.25$ the empirical distribution is relatively close to the triangular distribution whereas for values over $2.25$ it is close to the normal distribution \eqref{eqn:fnorm}. For $\eta$ in the interval $(1.25,2.25)$ the distribution is probably better explained by a mixture of the two distributions. The small offset from zero for small values of the distance is due to the variance of the monte carlo simulation.
\begin{figure}
\centerline{
{\psfrag{eta}[bc][bc]{$\eta$}
\psfrag{dW}[bc][bc]{$d_W$}
\includegraphics[width=0.5\textwidth]{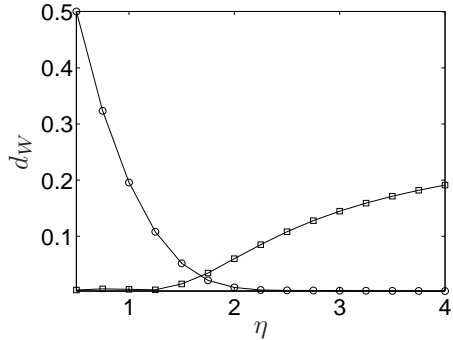}}}
\caption{Distance in terms of the Wasserstein metric between the triangular distribution and the empirical distribution (squares), and the normal distribution \eqref{eqn:fnorm} and the empirical distribution (circles).}
\label{fig:dist}
\end{figure}

\begin{figure}
\centerline{
{\psfrag{time}[bc][bc]{Time}
\psfrag{variance}[bc][bc]{Variance}
\includegraphics[width=0.5\textwidth]{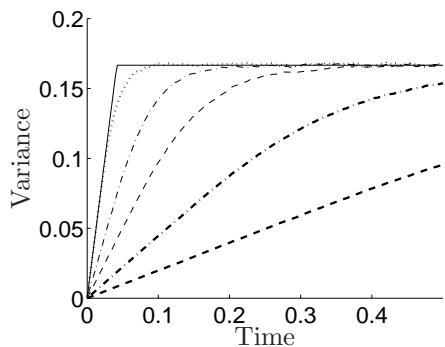}}}
\caption{The variance of $Z^{\sigmaoptk \eta}_t/\eta$ as a function of time where $\eta=0.50$ (dotted line), $\eta=0.75$ (thin dash-dotted line), $\eta=1.00$ (thin dashed line), $\eta=1.50$ (thick dash-dotted line) and $\eta=2.25$ (thick dashed line), together with the function $(t/0.5^2)\wedge(1/6)$ (solid line).}
\label{fig:var}
\end{figure}
From \eqref{eqn:pdf_Zeta} it is clear that it is possible to fix $\eta$ and instead of letting $\eta$ approach zero let $t$ approach infinity. To capture this we have plotted the variance of $Z^{\sigmaoptk \eta}_t/\eta$ as a function of $t$ for a couple of different values of $\eta$ (see Figure \ref{fig:var}). The constant $1/6$, that is the value of the variance of the triangularly distributed random variable, is also plotted in the figure. As expected it is seen that for low values of $\eta$ the limiting variance of $1/6$ is attained much faster than for higher values of $\eta$. From the argumentation above regarding high values of $\eta$ it is also clear that for low values of $t$ the distribution is approximately normal. Hence, the slope of the lines near zero is given by $1/\eta^2$, as is seen in the figure.

\end{document}